\documentclass[11pt,dvips,twoside,letterpaper]{article}
\usepackage{pslatex}
\usepackage{fancyhdr}
\usepackage{graphicx}
\usepackage{geometry}
\RequirePackage[T1]{fontenc}

\def\figurename{Figure} % Replace the colon that normally appears after the Figure number by a period.
\makeatletter
\renewcommand{\fnum@figure}[1]{\figurename~\thefigure.}
\makeatother

\def\tablename{Table} % Replace the colon that normally appears after the Figure number by a period.
\makeatletter
\renewcommand{\fnum@table}[1]{\tablename~\thetable.}
\makeatother

\usepackage{amsmath}
\usepackage{amssymb}
\usepackage{amsfonts}
\usepackage{amsthm,amscd}

\newtheorem{theorem}{Theorem}[section]
\newtheorem{lemma}[theorem]{Lemma}

\theoremstyle{definition}
\newtheorem{definition}[theorem]{Definition}

\theoremstyle{remark}
\newtheorem{remark}[theorem]{Remark}

\numberwithin{equation}{section}

\def\R{\mathbb R}
\def\E{\mathbb E}

\def\E{\mathbb E}

\begin{document}
%\vskip 0.4in
\title{\bfseries\scshape{Generalized backward doubly
stochastic differential equations driven by Lévy processes with non-Lipschitz coefficients \thanks{This works is partially support by Fellowship grant of AMMSI}}}
\author{\bfseries\scshape Auguste Aman\thanks{E-mail address: augusteaman5@yahoo.fr (Corresponding author)}\, and Jeam Marc Owo\thanks{E-mail address: owojmarc@hotmail.com}\\
UFR de Math\'{e}matiques et Informatique,\\ Universit\'{e} de Cocody, C\^{o}te d'Ivoire\\22 BP 582 Abidjan 22}

\date{}
\maketitle \thispagestyle{empty} \setcounter{page}{1}

% ------- [First Page Running Head] - place it immediately after title! ------
%\thispagestyle{fancy} \fancyhead{}
%\fancyhead[L]{\\{\thepage--\pageref{lastpage-01} (2008)}} % put \label{lastpage-xx} on the last page!
%\fancyhead[R]{  \\ {\sf }}
%\fancyfoot{}
%\renewcommand{\headrulewidth}{0pt}
%------------------------------------------------------------------------------

\begin{abstract}
We study the homogenization problem of semilinear reflected partial
We prove an existence and uniqueness result for generalized backward doubly
stochastic differential equations driven by Lévy processes with non-Lipschitz assumptions.
\end{abstract}


\vspace{.08in} \noindent \textbf{Keywords}: Backward doubly stochastic
differential equations; Lévy processes; non-Lipschitz coefficients, Teugel martingale

\section{Introduction}
Nonlinear backward stochastic differential equations (BSDEs in
short) have been introduced by Pardoux and Peng \cite{PP1}. The
original motivation for the study of this kind of equations was to
provide probabilistic interpretation for solutions of both
parabolic and elliptic semi linear partial differential equations
(see Pardoux and Peng \cite{PP2}, Peng  \cite{Pg}). Thanks to its link with the finance \cite{ElKaroui}, the stochastic control and stochastic game theory (see \cite{Hamadene} and references therein),  the theory of BSDEs quickly tooks a real enthusiasm since 1990.\\
Moreover, in order to give a probabilistic representation for a class of quasilinear stochastic partial differential
equations (SPDEs in short), Pardoux and Peng \cite{PardPeng}
considered a new kind of BSDEs, called backward doubly stochastic differential equations (BDSDEs
in short). There exist two different direction of stochastic integral driven respectively by two independent Brownian motion. The first integral is the well-know backward Itô integral and the second, the forward one. Following it, Bally and Matoussi \cite{BMat} gave the probabilistic
representation of the weak solutions to parabolic semilinear SPDEs in Sobolev spaces by
means of BDSDEs. Furthermore, Boufoussi et al. \cite{Boufsi} recommended a class of
generalized BDSDEs (GBDSDEs in short) which involved an other integral with
respect to an adapted continuous increasing process and gave the
probabilistic representation for stochastic viscosity solutions of semi-linear SPDEs with a
Neumann boundary condition. \\ Recently, Ren et al. \cite{Ren} showed the existence
and uniqueness of solutions to GBDSDEs driven by Teugels martingales associated
with Lévy process and gave probabilistic interpretation for solutions to a class of stochastic partial differential integral
equations (SPDIEs in short) with a nonlinear Neumann boundary condition. These results are obtained
with strong conditions on the coefficients, those are Lipschitz conditions
and monotony ones. Recently, N'zi and Owo \cite{MNJMO2} proved an existence and uniqueness
result of solutions for BDSDEs with non-Lipschitz conditions.

Inspired by this work, the aim of this paper is to extend the study of GBDSDEs driven by Lévy
processes introduced by in Ren et al. \cite{Ren}. We prove an existence and uniqueness
result in the non-Lipschitz case.

The rest of the paper is organized as follows. In section 2, we introduce some
preliminaries and notations. Section 3 is devoted to the proof of the existence
and uniqueness of the solutions to GBDSDEs driven by Lévy processes with non-Lipschitz coefficients.
\section{Preliminaries and notations}

Let $(\Omega, \mathcal{F},\mathbb{P})$ be a complete probability
space on which are defined all the processes considered and $T$ be
a
fixed final time.
\\
Let  $\{B_{t}; 0\leq t\leq T \}$ be a standard
Brownian motion, with values in $\mathbb{R}^{}$ and $%
\{L_{t}; 0\leq t\leq T \}$ be a $\mathbb{R}%
^{}$-valued Lévy process independent of $\{B_{t}; 0\leq t\leq T \}$ corresponding to a standard Lévy measure $\nu$ such that $\int_{\R}(1\wedge y)\nu(dy)<\infty$.
\\
Let $\mathcal{N}$ denote the class of $P$-null sets
of $\mathcal{F}$. For each $t \in [0,T]$, we define
\begin{equation*}
\mathcal{F}_{t}\overset{\Delta}{=}\mathcal{F}_{t}^{L} \vee \mathcal{F}%
_{t,T}^{B},
\end{equation*}
where for any process $\{\eta_{t}\}$ ; $\mathcal{F}_{s,t}^{\eta}=\sigma
\{\eta_{r}-\eta_{s}; s\leq r \leq t \} \vee \mathcal{N}$, $\mathcal{F}%
_{t}^{\eta}=\mathcal{F}_{0,t}^{\eta}$. \newline
Note that $\{\mathcal{F}_{t}^L, t\in [0,T]\}$ is an increasing filtration
and $\{\mathcal{F}_{t,T}^B, t\in [0,T]\}$ is a decreasing filtration, and
the collection $\{\mathcal{F}_{t}, t\in [0,T]\}$ is neither increasing nor
decreasing so that it does not constitute a filtration.
\\
Let $\ell^{2}$ denote
the set of real valued sequences $x=(x^{(i)})_{i\geq1}$ such that
$\displaystyle\|x\|^2=\sum_{i=1}^{\infty}|x^{(i)}|^2<\infty.$\\
We will denote by $\mathcal{M}^{2}(0,T,\ell^{2})$
the set of (class of $dP\otimes dt$ a.e. equal) $\ell^{2}$-valued process which satisfy
\begin{enumerate}
\item[(i)] $\|\varphi \|_{\mathcal{M}^{2}(\ell^{2})}^{2}=\mathbb{E}
\left(\int_{0}^{T}\|\varphi_{t}\|^2dt\right)< \infty.$

\item[(ii)] $\varphi_{t}$ is $\mathcal{F}_{t}$-measurable, for a.e. $t \in
[0,T].$
\end{enumerate}
Similarly, $\mathcal{S}^{2}(0,T)$
stands for the set of real valued random processes which satisfy:
\begin{enumerate}
\item[(i)] $\|\varphi \|_{\mathcal{S}^{2}}^{2}=\mathbb{E}\left(%
\underset{0\leq t\leq T}{\sup}|\varphi_{t}|^{2}\right)< \infty$

\item[(ii)] $\varphi_{t}$ is $\mathcal{F}_{t}$-measurable, for any $t \in
[0,T].$
\end{enumerate}
In the sequel, let $\{A_{t}; 0\leq t\leq T \}$
be a continuous and increasing real valued process such that
$A_{t}$ is $\mathcal{F}_{t}$-measurable, for any $t \in[0,T]$ and $A_0=0.$
\\
Let $\mathcal{A}^{2}(0,T)$ denote
the set of (class of $dP\otimes dA_t$ a.e. equal) real
valued measurable random processes $\{\varphi_{t}; 0\leq t\leq T \}$
such that $\displaystyle\mathbb{E}%
\left(\int_{0}^{T}|\varphi_{t}|^{2} dA_t\right)< \infty.$\\
 We will denote by
 $\mathcal{E}(0,T)=\big(\mathcal{S}^{2}(0,T)
 \cap\mathcal{A}^{2}(0,T)\big)\times\mathcal{M}^{2}(0,T,\ell^{2})$
the set of $\mathbb{R}%
 \times \mathrm{\ell}^{2}$-valued processes $(Y,Z)$
defined on $\Omega \times [0,T]$ which satisfy condition (ii) as above and such that
$$
\|(Y,Z)\|_{\mathcal{E}}^{2}=\mathbb{E}\left(%
\underset{0\leq t\leq T}{\sup} |Y_{t}|^{2}
+\int_{0}^{T}|Y_{s}|^{2} dA_s+
\int_{0}^{T}\|Z_{s}\|^{2} ds\right)< \infty.
$$
$\mathcal{E}(0,T)$
endowed with the norm $\|.\|_{\mathcal{E}}^{}$ is a Banach space.
\\
Let denote by $(H^{(i)})_{i\geq1}$ the Teugels Martingale
associated with the Lévy process  $\{L_{t}; 0\leq t\leq T \}$. More precisely
$$H_{t}^{(i)}=c_{i,i}T_{t}^{(i)}+c_{i,i-1}T_{t}^{(i-1)}+...+c_{i,1}T_{t}^{(1)},$$
where $T_{t}^{(i)}=L_{t}^{(i)}-\E(L_{t}^{(i)})=L_{t}^{(i)}-t\E(L_{1}^{(i)})$
for all $i\geq1$ and $L_{t}^{(i)}$  are power jump processes so that
$L_{t}^{(1)}=L_{t}$ and $\displaystyle L_{t}^{(i)}=\sum_{0<s\leq t}( \Delta L_{s})^i$
for $i\geq2$, with $\displaystyle L_{t^{-}}=\lim_{s\nearrow t}L_s$ and $\Delta L_{s}=  L_{s^{}}-L_{s^{-}}$. \\
Nualart and Schoutens have proved in \cite{NSc} that
the coefficients $c_{i,k}$ correspond to the orthonormalization
of the polynomials $1,\ x, \ x^2,\cdot\cdot\cdot$  with respect to the measure
$\mu(dx)=x^2\nu(dx)+\sigma^2\delta_0(dx):$
$$q_{i}(x)=c_{i,i}x^{i-1}+c_{i,i-1}x^{i-2}+\cdot\cdot\cdot+c_{i,1}.$$
%$$p_{i}(x)=c_{i,i}x^i+c_{i,i-1}x^i+...+c_{i,1}x.$$
The martingale $(H^{(i)})_{i\geq1}$ can be chosen to be
pairwise strongly orthonormal martingale. \\That is for all $i,j$,
\ $\displaystyle\langle H^{(i)},H^{(j)}\rangle_{t}=\delta_{ij}t$.
\begin{remark}
If $\mu$ only has mass at $1$, we are in the Poisson case $N_t$ with parameter $\lambda>0$; here $H_t^{(1)}=\frac{N_t-\lambda t}{\lambda}$
 and $H^{(i)}=0, \ \ i=2,\ 3,\cdot\cdot\cdot$. This case is degenerate in this Lévy framework.
\end{remark}
\begin{definition}
A pair $(Y,Z):\Omega \times [0,T]\rightarrow \mathbb{R} \times
\ell^{2}$  of processes is called solution \\of GBDSDE$(\xi, f,g, h, A)$
driven by Lévy processes if $(Y,Z)
\in \mathcal{E}(0,T)$ such that
\begin{eqnarray}\label{eq0}
Y_t&=&\xi+\int_{t}^{T}f(s,Y_{s^{-}},Z_s)ds+
\int_{t}^{T}h(s,Y_{s^{-}})dA_{s}+
\int_{t}^{T}g(s,Y_{s^{-}},Z_s)\overleftarrow{dB_{s}}\notag\\
&&-\sum_{i=1}^{\infty}\int_{t}^{T}Z_s^{(i)}dH_{s}^{(i)},\ \ t\in[0,T].
\end{eqnarray} Here the integral with respect to $\{B_{t}\}$ is
the classical backward Itô integral (see Kunita \cite{Kunita})
and the integral with respect to $\{(H_{t}^{(i)})_{i\geq 1}\}$
is a standard forward Itô-type semi martingale integral (see Gong \cite{Gg}).
\end{definition}
\noindent First, let us recall the extension of the well-known
Itô formula on which depend strongly our results. Its proof follows the same program as Lemma 2.5 in \cite{Boufsi} or Lemma 1.3 in \cite{PardPeng}.
\begin{lemma}
Let $\alpha,\beta$ and $\gamma$ in $\mathcal{S}^{2}(0,T)$, $\eta\in\mathcal{A}^{2}(0,T)$ and
 $\zeta\in\mathcal{M}^{2}(0,T,\ell^{2})$ such that
\begin{eqnarray*}\label{}
\alpha_t=\alpha_0+\int_{t}^{T}\beta_sds+
\int_{t}^{T}\eta_sdA_{s}+\int_{t}^{T}\gamma_sdB_{s}
-\sum_{i=1}^{\infty}\int_{t}^{T}\zeta_s^{(i)}dH_{s}^{(i)},\ \ t\in[0,T]\notag.
\end{eqnarray*}Then
\begin{eqnarray*}\label{}
|\alpha_t|^{2}&=&|\xi|^{2}+2\int_{t}^{T}\alpha_s\beta_sds+
2\int_{t}^{T}\alpha_s\eta_sdA_{s}+2\int_{t}^{T}\alpha_s\gamma_sdB_{s}\\
&&-2\sum_{i=1}^{\infty}\int_{t}^{T}\alpha_s\zeta_s^{i}dH_{s}^{(i)}
+\int_{t}^{T}|\gamma_s|^{2}ds
-\sum_{i,j=1}^{\infty}\int_{t}^{T}\zeta_s^{i}\zeta_s^{j}d[H_{s}^{(i)},H_{s}^{(j)}].
\end{eqnarray*}
Note that $\displaystyle\left(\int_{t}^{T}\alpha_s\gamma_sdB_{s}\right)_{0\leq t\leq T}$,
 $\displaystyle\left(\int_{0}^{t}\alpha_s\zeta_s^{(i)}dH_{s}^{(i)}\right)_{0\leq t\leq T}$
for all $i\geq1$ and
$\displaystyle\left(\int_{0}^{t}\zeta_s^{(i)}\zeta_s^{(j)}
d[H_{s}^{(i)},H_{s}^{(j)}]\right)_{0\leq t\leq T}$
for $i\not=j$ are uniformly integrable martingale and
$\displaystyle\langle H^{(i)},H^{(j)}\rangle_{t}=\delta_{ij}t$, we have
\begin{eqnarray*}\label{}
\E|\alpha_t|^{2}&=&\E|\alpha_0|^{2}+2\E\int_{t}^{T}\alpha_s\beta_sds+
2\E\int_{t}^{T}\alpha_s\eta_sdA_{s}+\E\int_{t}^{T}|\gamma_s|^{2}ds
\\
&&-\E\left(\int_{t}^{T}\sum_{i=1}^{\infty}|\zeta_s^{(i)}|^{2}ds\right),\ \ t\in[0,T]\notag.
\end{eqnarray*}
\end{lemma}
Next, let us recall the existence and uniqueness result on GBDSDE$(\xi, f,g, h, A)$
in the Lipschitz and monotony context. This work is due to Ren et al. \cite{Ren}.
They use the following assumptions:
\begin{description}
%\item \textbf{(H1)}\ \ There exists a constant $k>0$ such
%that for all $t\in[0,T],\ A_t\geq kt$.

\item \textbf{(A1)}\ \ The terminal value $\xi \in \mathrm{L}^{2}(\Omega,
\mathcal{F}_{T}, \mathbb{P}, \mathbb{R})$ such that for all $\lambda>0$
$$\E (e^{\lambda A_T}|\xi|^2)<\infty.$$
\noindent \item \textbf{(A2)}\ \ The coefficients $f, g:\Omega \times [0,T]\times
\mathbb{R} \times \ell^{2}\rightarrow
\mathbb{R}$ and $h:\Omega \times [0,T]\times \mathbb{R}
\rightarrow \mathbb{R}$ satisfy, for some $\beta_1\in\R$,
$K>0$, $0 < \alpha < 1$ and $\beta_2<0$, three $\mathcal{F}_t$-adapted processes
$\{f_t, g_t, h_t:0\leq t\leq T\}$ with values in $[1,\infty[$
and for all $(t,y,z)\in [0,T]\times
\mathbb{R} \times \ell^{2}$, $\lambda >0$
\begin{itemize}
  \item [(i)] $f(.,y,z), g(.,y,z)$ and $h(.,y)$ are progressively measurable,
  \item [(ii)] $\left\{
             \begin{array}{ll}
               |f(t,y,z)|\leq f_t+ K(|y|+\|z\|) \vspace{0.1cm}& \hbox{} \\
               |g(t,y,z)|\leq g_t+ K(|y|+\|z\|) \vspace{0.1cm}& \hbox{} \\
               |h(t,y)|\leq h_t+K|y| & \hbox{}
             \end{array}
           \right.$
  \item [(iii)] $\displaystyle\E (\int_{0}^{T}e^{\lambda A_t}f_t^2dt+
  \int_{0}^{T}e^{\lambda A_t}g_t^2dt+\int_{0}^{T}e^{\lambda A_t}h_t^2dA_t)<\infty$
  \item [(iv)] $\langle y-y',f(t,y,z)-f(t,y',z)\rangle  \leq \beta_1\mid
y-y'\mid^{2}  $
  \item [(vi)] $\mid f(t,y,z)-f(t,y,z') \mid^{2} \leq K\| z-z'\|^{2}$
  \item [(vii)]$\langle y-y',h(t,y)-h(t,y')\rangle  \leq \beta_2\mid
y-y'\mid^{2}  $
  \item [(vii)] $\mid g(t,y,z)-g(t,y',z') \mid^{2} \leq K\mid
y-y'\mid^{2}+\alpha \|z-z'\|^{2}$
  \item [(viii)] $y\mapsto(f(t,y,z), g(t,y,z), h(t,y))$ is continuous for all $z$, $(\omega, t)$.
\end{itemize}

\noindent \item \textbf{(A3)}\ \ $\mid f(t,y,z)-f(t,y',z) \mid^{2}+
\mid h(t,y)-h(t,y') \mid^{2} \leq K\mid y-y' \mid^{2}$.
\end{description}
\begin{lemma}[Ren et al. \cite{Ren}]\label{lm0}
Under the assumptions \textbf{(A1)}, \textbf{(A2)} and \textbf{(A3)}, \\
the GBDSDE$(\xi, f,g, h, A)$ has a unique solution
\end{lemma}
\section{Existence and uniqueness result in non-Lipschtz case}
In order to attain the solution of GBDSDE $(\xi, f,g, h, A)$,
we stand the following assumptions.
The coefficients $f, g:\Omega \times [0,T]\times
\mathbb{R} \times \ell^{2}\rightarrow
\mathbb{R}$, $h:\Omega \times [0,T]\times \mathbb{R}
\rightarrow \mathbb{R}$ and the terminal value $\xi$ satisfy:
\begin{description}
%\item \textbf{(H1)}\ \ There exists a constant $k>0$ such
%that for all $t\in[0,T],\ A_t\geq kt$.

\item \textbf{(H1)}\ \ $f(.,y,z), g(.,y,z)$ and $h(.,y)$ are progressively
measurable such that

 $
   \begin{array}{ll}
 \displaystyle 0<\E \left(\int_{0}^{T}\left|f(s,0,0)\right|^2ds+\int_{0}^{T
}\left|h(s,0)\right|^2dA_s+\int_{0}^{T}
\left|g(s,0,0)\right|^2ds\right)<\infty. & \hbox{}
   \end{array}
 $

\item \textbf{(H2)}
For some $K>0$ and three $\mathcal{F}_t$-adapted processes
$\{f_t, g_t, h_t:0\leq t\leq T\}$ with values in $[1,\infty[$
and for all $(t,y,z)\in[0,T]\times
\mathbb{R} \times \ell^{2}$, $\lambda >0$

  $\displaystyle\left\{
             \begin{array}{ll}
               |f(t,y,z)|\leq f_t+ K(|y|+\|z\|) \vspace{0.2cm}& \hbox{} \\
               |g(t,y,z)|\leq g_t+ K(|y|+\|z\|) \vspace{0.2cm}& \hbox{} \\
               |h(t,y)|\leq h_t+K|y| \vspace{0.2cm}& \hbox{}\\
               \displaystyle\E \left(\int_{0}^{T}e^{\lambda A_t}f_t^2dt+
  \int_{0}^{T}e^{\lambda A_t}g_t^2dt+\int_{0}^{T}e^{\lambda A_t}h_t^2dA_t\right)<\infty & \hbox{}
             \end{array}
           \right.$

\noindent \item \textbf{(H3)}\ \ For some $\beta<0$ and for all
$y_{1},\ y_{2}\in \mathbb{R}%
$ and $t\in [0,T]$, $$\langle y_{1}-y_{2},h(t,y_{1})-h(t,y_{2})\rangle  \leq \beta\mid
y_{1}-y_{2}\mid^{2}  $$

\noindent \item \textbf{(H4)}\ \ For all $(y_{1},z_{1}),\ (y_{2},z_{2})\in \mathbb{R}%
 \times \mathrm{\ell}^{2}$ and $t\in [0,T]$,

$\left\{
\begin{array}{ll}
|h(t,y_{1})-h(t,y_{2})|  \leq K\mid
y_{1}-y_{2}\mid\\
\mid f(t,y_{1},z_{1})-f(t,y_{2},z_{2}) \mid^{2} \leq \rho(t,\mid
y_{1}-y_{2}\mid^{2})+C\|z_{1}-z_{2}\|^{2} \vspace{0.2cm} &
\\
\mid g(t,y_{1},z_{1})-g(t,y_{2},z_{2}) \mid^{2} \leq \rho(t,\mid
y_{1}-y_{2}\mid^{2})+\alpha \| z_{1}-z_{2}\|^{2} &
\end{array}
\right.,
$

where $C > 0$ and $0 < \alpha < 1$ are two constants and $\rho: [0,T]\times%
\mathbb{R}^+ \rightarrow \mathbb{R}^+$ satisfies:

\begin{itemize}
\item[(i)] for fixed $t\in[0,T]$, \ $\rho(t,.)$ is a concave
and non-decreasing function such that $\rho(t,0)=0.$

\item[(ii)] for fixed $u$, $\int_0^T\rho(t,u)dt<+\infty$.

\item[(iii)] for any $M>0$, the following ODE
\begin{equation*}
\left\{
\begin{array}{ccc}
u^{\prime } & = & -M\rho (t,u) \\
u(T) & = & 0%
\end{array}%
\right.
\end{equation*}%
has a unique solution $u(t)\equiv 0,\ \ t\in \lbrack 0,T]$.
\end{itemize}

\noindent \item \textbf{(H5)} \ \ $\xi \in \mathrm{L}^{2}(\Omega,
\mathcal{F}_{T}, \mathbb{P}, \mathbb{R})$ such that for all $\lambda>0$
$$\E (e^{\lambda A_T}|\xi|^2)<\infty.$$
\end{description}
%\section{Existence and uniqueness theorem for BDSDEs with
%non-Lipschitz\\
%coefficients}
\noindent Under above assumptions, we now construct an approximate sequence using a Picard-type iteration with the help of
Lemma \ref{lm0}. Let
$Y_t^0=0$ and $\left(
Y^n,Z^n\right)_{n\geq 1}$ be a sequence in $\mathcal{E}^{2}(0,T)$ defined recursively
by
\begin{eqnarray}  \label{eq3}
Y_{t}^{n}&=&\xi +\int_{t}^{T }f(s,Y_{s}^{n-1},Z_{s}^{n})ds+ \int_{t}^{T
}h(s,Y_{s}^{n})dA_s+ \int_{t}^{T
}g(s,Y_{s}^{n-1},Z_{s}^{n})\overleftarrow{dB_{s}}\notag\\
&&-\sum_{i=1}^{\infty}\int_{t}^{T}Z_s^{n(i)}dH_{s}^{(i)}.
\end{eqnarray}
Indeed, for each $n\geq 1$ and fixed $Y^{n-1}$ in $\mathcal{S}^{2}(0,T)$,
BDSDE \eqref{eq3} satisfies assumptions \textbf{(A1)}, \textbf{(A2)} and \textbf{(A3)}. So, by Lemma \ref{lm0}, the BDSDE \eqref{eq3} has a unique solution $\left(
Y^{n},Z^{n}\right)\in \mathcal{E}^{2}(0,T)$. \newline
Our purpose is to prove that the sequence $\left( Y^n,Z^n\right)_{n\geq 0}$
converges in $\mathcal{E}^{2}(0,T)$ to the unique solution of BDSDEs \eqref{eq0}. We
begin with some preliminary results.

\begin{lemma}
\label{l2}Let $\textbf{(H1)}$, $\textbf{(H3)}$ and $\textbf{(H4)}$ be satisfied.
Then for all $0\leq t\leq T$, $n, m\geq 1$, we have
\begin{eqnarray*}
{\mathbb{E}}\left| Y_{t}^{n+m}-Y_{t}^{n}\right| ^{2} \leq e^{\frac{CT}{%
1-\alpha}}\left(\frac{1-\alpha}{C}+1\right)\int_{t}^{T }\rho(s,{\mathbb{E}}%
\left| Y_{s}^{n+m-1}-Y_{s}^{n-1}\right| ^{2})ds.
\end{eqnarray*}
\end{lemma}

\begin{proof} In view of It\^o's formula, we have
\begin{eqnarray*}
&&{\E}\left|Y_{t}^{n+m}-Y_{t}^{n}\right|^{2}+{\E}\int_{t}^{T }\|Z_{s}^{n+m}-Z_{s}^{n}\|^{2}ds \\
&=&2{\E}\int_{t}^{T }\left\langle
Y_{s}^{n+m}-Y_{s}^{n},f(s,Y_{s}^{n+m-1},Z_{s}^{n+m})-f(s,Y_{s}^{n-1},Z_{s}^{n})\right\rangle
ds\\&&+2{\E}\int_{t}^{T }\left\langle
Y_{s}^{n+m}-Y_{s}^{n},h(s,Y_{s}^{n+m})-h(s,Y_{s}^{n})\right\rangle
dA_s\\&&+{\E}\int_{t}^{T }\left|
g(s,Y_{s}^{n+m-1},Z_{s}^{n+m})-g(s,Y_{s}^{n-1},Z_{s}^{n})\right|
^{2}ds .
\end{eqnarray*}
In view of { \bf (H3)}, and Young's inequality  $2ab\leq
\frac{1}{\theta}a^2+\theta b^2 $, for any $\theta>0$, we have
\begin{eqnarray*}\label{cz}
&&{\E}\left|Y_{t}^{n+m}-Y_{t}^{n}\right| ^{2}+
{\E}\int_{t}^{T }\|Z_{s}^{n+m}-Z_{s}^{n}\| ^{2}ds +2|\beta|{\E}\int_{t}^{T
}\left|Y_{s}^{n+m}-Y_{s}^{n}\right|^2 dA_s\notag\\
&\leq&\frac{1}{\theta}{\E}\int_{t}^{T
}\left|Y_{s}^{n+m}-Y_{s}^{n}\right|^2 ds+(\theta+1){\E}\int_{t}^{T
} \rho(s,\left|Y_{s}^{n+m-1}-Y_{s}^{n-1}\right|^2) ds\\&&+(\theta
C+\alpha){\E}\int_{t}^{T }\|Z_{s}^{n+m}-Z_{s}^{n}\|^2 ds\notag.
\end{eqnarray*}
Choosing $\theta=\frac{1-\alpha}{C}>0$, it follows from Gronwall's
inequality and Jensen's inequality that
\begin{eqnarray*}
{\E}\left| Y_{t}^{n+m}-Y_{t}^{n}\right| ^{2} \leq
e^{\frac{CT}{1-\alpha}}\left(\frac{1-\alpha}{C}+1\right)\int_{t}^{T
}\rho(s,{\E}\left| Y_{s}^{n+m-1}-Y_{s}^{n-1}\right| ^{2})ds.
\end{eqnarray*}
\end{proof}

\begin{lemma}
\label{l3} Let $\textbf{(H1)}$, $\textbf{(H3)}$ and $\textbf{(H4)}$ be
satisfied. Then, there exists $T_1 \in [0,T[$ and a constant
$M_1\geq 0$ such that for all $t\in[T_1, T]$, \ each $n\geq 1$,\
${\mathbb{E}}\left|Y_{t}^{n}\right| ^{2} \leq M_1$.
\end{lemma}

\begin{proof} In view of It\^o's formula, we have
\begin{eqnarray*}
&&{\E}\left|Y_{t}^{n}\right| ^{2}+{\E}\int_{t}^{T }\|Z_{s}^{n}\|^{2}ds \\
&=&{\E}\left|\xi\right| ^{2}+2{\E}\int_{t}^{T }\left\langle
Y_{s}^{n},f(s,Y_{s}^{n-1},Z_{s}^{n})\right\rangle
ds+2{\E}\int_{t}^{T }\left\langle
Y_{s}^{n},h(s,Y_{s}^{n})\right\rangle
dA_s\\&&+{\E}\int_{t}^{T }\left| g(s,Y_{s}^{n-1},Z_{s}^{n})\right|
^{2}ds .
\end{eqnarray*}
By virtue of  {\bf (H3)}, {\bf (H4)} and Young's inequality
$2ab\leq
\frac{1}{\theta}a^2+\theta b^2 ,$
 for any $\theta>0,$  we have
\begin{eqnarray*}
2\left\langle
Y_{s}^{n},f(s,Y_{s}^{n-1},Z_{s}^{n})\right\rangle\!&\leq&\!\!
\frac{1}{\theta}\left|Y_{s}^{n}\right|^2+\theta\left|f(s,Y_{s}^{n-1},Z_{s}^{n})\right|^2
\\\!&\leq&\!\!\!\frac{1}{\theta}\left|Y_{s}^{n}\right|^2\!+\!
2\theta\rho(s,\left|Y_{s}^{n-1}\right|^2)\!\!+\! 2\theta
C\|Z_{s}^{n}\|^2\!+\!2\theta\left|f(s,0,0)\right|^2,
\end{eqnarray*}
\begin{eqnarray*}
2\left\langle
Y_{s}^{n},h(s,Y_{s}^{n})\right\rangle\!&\leq&\!\!
2\beta\left|Y_{s}^{n}\right|^2+2\left\langle
Y_{s}^{n},h(s,0)\right\rangle
\\\!&\leq&\!\!\!-|\beta|\left|Y_{s}^{n}\right|^2\!+\!
\frac{1}{|\beta|}\left|h(s,0)\right|^2,\hspace{5cm}
\end{eqnarray*}
$$\left| g(s,Y_{s}^{n-1},Z_{s}^{n})\right|
^{2}\leq(1+\theta)\rho(s,\left|Y_{s}^{n-1}\right|^2)+
(1+\theta)\alpha\|Z_{s}^{n}\|^2+(1+\frac{1}{\theta})\left|g(s,0,0)\right|^2.$$
Therefore,
\begin{eqnarray*}
&&{\E}\left|Y_{t}^{n}\right| ^{2}+[1-2\theta
C-(1+\theta)\alpha]{\E}\int_{t}^{T
}\|Z_{s}^{n}\|^2 ds +|\beta|{\E}\int_{t}^{T
}\left|Y_{s}^{n}\right|^2 dA_s \\
&\leq&{\E}\left|\xi\right| ^{2}+\frac{1}{\theta}{\E}\int_{t}^{T
}\left|Y_{s}^{n}\right|^2 ds+(3\theta+1)\int_{t}^{T }
\rho(s,{\E}\left|Y_{s}^{n-1}\right|^2) ds\\&&+{\E}\int_{t}^{T
}[2\theta\left|f(s,0,0)\right|^2+
(1+\frac{1}{\theta})\left|g(s,0,0)\right|^2]ds+\frac{1}{|\beta|}{\E}\int_{t}^{T
}\left|h(s,0)\right|^2dA_s.
\end{eqnarray*}
We choose $\theta=\frac{1-\alpha}{2C+\alpha}>0$, then
\begin{eqnarray*}
{\E}\left|Y_{t}^{n}\right| ^{2} &\leq&{\E}\left|\xi\right|
^{2}\!+\!\frac{2C\!+\!\alpha}{1\!-\!\alpha}{\E}\int_{t}^{T
}\!\!\left|Y_{s}^{n}\right|^2
ds\!+\!\left(3\frac{1-\alpha}{2C\!+\!\alpha}\!+\!1\right)\int_{t}^{T
}\!\!
\rho(s,{\E}\left|Y_{s}^{n-1}\right|^2) ds\\
&&+{\E}\int_{t}^{T
}\left[\frac{2(1-\alpha)}{2C+\alpha}\left|f(s,0,0)\right|^2+
(\frac{1+2C}{1-\alpha})\left|g(s,0,0)\right|^2\right]ds\\&&+\frac{1}{|\beta|}{\E}\int_{t}^{T
}\left|h(s,0)\right|^2dA_s.
\end{eqnarray*}
Now, in view of Gronwall's inequality, we derive
\begin{eqnarray}\label{g1}
{\E}\left|Y_{t}^{n}\right| ^{2} &\leq&\mu_t^1+\left(3\frac{1-\alpha}{2C+\alpha}+1\right)e^{\frac{(2C+\alpha)T}{1-\alpha}}\int_{t}^{T }
\rho(s,{\E}\left|Y_{s}^{n-1}\right|^2) ds
\end{eqnarray}
where
\begin{eqnarray*}&&\mu_t^1=e^{\frac{(2C+\alpha)T}{1\!-\!\alpha}}\left({\E}\left|\xi\right|
^{2}+{\E}\int_{t}^{T
}\left[\frac{2(1-\alpha)}{2C\!+\!\alpha}\left|f(s,0,0)\right|^2\!+\!
(\frac{1+2C}{1-\alpha})\left|g(s,0,0)\right|^2\right]ds\right.\\&&\ \hspace{2cm}\ \ \ \ \ \ \ \ \ \ \ \ \ \left.+\frac{1}{|\beta|}{\E}\int_{t}^{T
}\left|h(s,0)\right|^2dA_s\right),
\end{eqnarray*}
Let
\begin{eqnarray}
\, M=\max\left\{\left(3\frac{1-\alpha}{2C+\alpha}+1\right)
e^{\frac{(2C+\alpha)T}{1-\alpha}},\left(\frac{1-\alpha}{C}+1\right)
e^{\frac{CT}{1-\alpha}}\right\}>0.\label{M}
\end{eqnarray} and
\begin{eqnarray*}
&&M_1=2\mu_0^1=2e^{\frac{(2C+\alpha)T}{1-\alpha}}\left({\E}\left|\xi\right|
^{2}+{\E}\int_{0}^{T
}\left[\frac{2(1-\alpha)}{2C+\alpha}\left|f(s,0,0)\right|^2\right.\right.\\
&&\ \ \ \ \ \ \ \ \ \ +\left.\left.
(\frac{1+2C}{1-\alpha})\left|g(s,0,0)\right|^2\right]ds+\frac{1}{|\beta|}{\E}\int_{0}^{T
}\left|h(s,0)\right|^2dA_s\right)\geq
0.\end{eqnarray*}
 By virtue of {\bf (H4)},
 $\displaystyle\int_{0}^{T }
\rho(s,M_1) ds<+\infty,$ so we can find $T_1 $ such that
$$\displaystyle\int_{T_1}^{T } \rho(s,M_1) ds\leq \frac{\mu_0^1}{M}.$$
Now, we complete the proof as in N'zi and Owo \cite{MNJMO2}.

\end{proof}
\noindent With the help of the above Lemmas, we can now prove
existence and uniqueness which is our main result.
\begin{theorem}
\label{te} Let \textbf{(H1)}-\textbf{(H5)} be satisfied.
Then the equation \eqref{eq0} has an unique solution $(Y,Z) \in \mathcal{E}%
^{2}(0,T)$.
\end{theorem}

\begin{proof}
{Existence.} For all $n\geq 1$, and $t\in [0,T]$, we let
$$\phi_0(t)=M\int_{t}^{T } \rho(s,M_1)ds \ \ \text{and}\ \
\phi_{n+1}(t)=M\int_{t}^{T } \rho(s,\phi_{n}(s)) ds.$$
N'zi and Owo proved in \cite{MNJMO2} that $(\phi_n(t))_{n\geq 0}$ is non-increasing and converges
uniformly to $0$ for all $t\in [T_1,T]$.
\\
In view of Lemmas \ref{l2} and \ref{l3}, we conclude as in \cite{MNJMO2} that
for all $t\in[T_1,T]$, $n,\ m \geq 1$,
\begin{eqnarray}\label{3i}{\E}\left|
Y_{t}^{n+m}-Y_{t}^{n}\right| ^{2} \leq\phi_{n-1}(t)\leq M_1.
\end{eqnarray}
Using It\^o's formula, we deduce from assumptions
{\bf (H3)}, {\bf (H4)} and Young's inequality
$2ab\leq \frac{1}{\theta}a^2+\theta b^2,\ \theta>0$, that for all $t\in [T_1,T]$
\begin{eqnarray*} &&\left|Y_{t}^{n+m}-Y_{t}^{n}\right|^{2}-(\theta
C+\alpha)\int_{t}^{T }\|Z_{s}^{n+m}-Z_{s}^{n}\| ^{2}ds+2|\beta|\int_{t}^{T
}\left|Y_{s}^{n+m}-Y_{s}^{n}\right|^2 dA_s \\
&\leq&\frac{1}{\theta}\int_{t}^{T
}\left|Y_{s}^{n+m}-Y_{s}^{n}\right|^2 ds+(\theta+1)\int_{t}^{T
} \rho(s,\left|Y_{s}^{n+m-1}-Y_{s}^{n-1}\right|^2) ds\\
&&+2\int_{t}^{T }\left\langle
Y_{s}^{n+m}-Y_{s}^{n},(g(s,Y_{s}^{n+m-1},Z_{s}^{n+m})-g(s,Y_{s}^{n-1},Z_{s}^{n}))
\overleftarrow{dB_{s}}\right\rangle\\
&&-2\sum_{i,j=1}^{\infty}\int_{t}^{T }\left\langle
Y_{s}^{n+m}-Y_{s}^{n},Z_{s}^{n+m(i)}-Z_{s}^{n(i)}
\right\rangle dH_s^{(i)}-
\sum_{i,j=1}^{\infty}\int_{t}^{T}Z_s^{i}Z_s^{j}d[H_{s}^{i},H_{s}^{j}].
\end{eqnarray*}
Note that $\displaystyle\left(\int_{t}^{T }\left\langle
Y_{s}^{n+m}-Y_{s}^{n},(g(s,Y_{s}^{n+m-1},Z_{s}^{n+m})-g(s,Y_{s}^{n-1},Z_{s}^{n}))
\overleftarrow{dB_{s}}\right\rangle\right)_{0\leq t\leq T}$,
\\
$\displaystyle\left(\int_{t}^{T }\left\langle
Y_{s}^{n+m}-Y_{s}^{n},Z_{s}^{n+m(i)}-Z_{s}^{n(i)}
\right\rangle dH_s^{(i)}\right)_{0\leq t\leq T}$
for all $i\geq1$ and
$\displaystyle\left(\int_{t}^{T}Z_s^{i}Z_s^{j}d[H_{s}^{i},H_{s}^{j}]\right)_{0\leq t\leq T}$
for $i\not=j$ are uniformly integrable martingale.\\
Therefore, taking expectation and Jensen inequality, we obtain from inequality \eqref{3i},
\begin{eqnarray*}
&&{\E}\left|Y_{t}^{n+m}-Y_{t}^{n}\right| ^{2}+(1-\theta
C-\alpha){\E}\int_{t}^{T }\|Z_{s}^{n+m}-Z_{s}^{n}\| ^{2}ds +2|\beta|{\E}\int_{t}^{T
}\left|Y_{s}^{n+m}-Y_{s}^{n}\right|^2 dA_s\\
&\leq&\frac{1}{\theta}{\E}\int_{t}^{T
}\left|Y_{s}^{n+m}-Y_{s}^{n}\right|^2 ds+(\theta+1)\int_{t}^{T
} \rho(s,{\E}\left|Y_{s}^{n+m-1}-Y_{s}^{n-1}\right|^2) ds\\&\leq&\frac{1}{\theta}\int_{t}^{T
}\phi_{n-1}(s) ds+\frac{\theta+1}{M}\phi_{n-1}(t).
\end{eqnarray*}
Thus, choosing $\theta=\frac{1-\alpha}{2C}$, we get
\begin{eqnarray*}
&&\sup_{T_1\leq t\leq T }\left({\E}\left|Y_{t}^{n+m}-Y_{t}^{n}\right| ^{2}\right)+\frac{1-\alpha}{2}{\E}\int_{T_1}^{T }\|Z_{s}^{n+m}-Z_{s}^{n}\| ^{2}ds +
2|\beta|{\E}\int_{T_1}^{T
}\left|Y_{s}^{n+m}-Y_{s}^{n}\right|^2 dA_s\ \ \ \ \ \ \ \\
&\leq&(\frac{T-T_1}{\theta}+\frac{\theta+1}{M})\phi_{n-1}(T_1).
\end{eqnarray*}
Now, in view of this inequality, we deduce by Burkhölder-Davis-Gundy's inequality that
\begin{eqnarray*}
{\E}\left(\sup_{T_1\leq t\leq T }\left|Y_{t}^{n+m}-Y_{t}^{n}\right| ^{2}\right)+
{\E}\int_{T_1}^{T }\|Z_{s}^{n+m}-Z_{s}^{n}\| ^{2}ds +
{\E}\int_{T_1}^{T
}\left|Y_{s}^{n+m}-Y_{s}^{n}\right|^2 dA_s
\leq K\phi_{n-1}(T_1),
\end{eqnarray*}where $K$ is positive constant depending on
$C$, $T_1$, $T$, $\alpha$,  $|\beta|$ and $M$.
\\
Since $\phi_{n}(t)\ \rightarrow \ 0$,\ for all $t\in[T_1,T]$, as
$n \ \rightarrow \ \infty$, it follows that $(Y^n,Z^n)$ is a Cauchy
sequence in $\mathcal{E}^{2}(T_1,T)$.  Now, set
$$Y=\underset{n\rightarrow +\infty}{\lim}Y^{n},
 \ \ \ Z=\underset{n\rightarrow +\infty}{\lim}Z^{n}.$$
Then, as $\mathcal{E}^{2}(T_1,T)$ is a Banach space, $(Y,Z)\in\mathcal{E}^{2}(T_1,T)$.
\\
Passing to the limit in \eqref{eq3}, we prove that
 $(Y,Z)$ satisfies the BDSDE \eqref{eq0} on $[T_1,T]$.
\\
If $T_1=0$, then we have proved the existence result.
If $T_1\neq 0$, we consider the following equation
\begin{eqnarray}\label{2e}
Y_t&=&Y_{T_1}+\int_{t}^{T_1}f(s,Y_{s^{-}},Z_s)ds+
\int_{t}^{T_1}h(s,Y_{s^{-}})dA_{s}+
\int_{t}^{T_1}g(s,Y_{s^{-}},Z_s)\overleftarrow{dB_{s}}\notag\\
&&-\sum_{i=1}^{\infty}\int_{t}^{T_1}Z_s^{(i)}dH_{s}^{(i)},\ \ t\in[0,T_1].
\end{eqnarray}
We construct the Picard approximate sequence of equation
\eqref{2e}, as in \eqref{eq3}. Using the same procedure as in the
proof of Lemmas \ref{l2} and Lemma \ref{l3}, for all
$t\in[T_1,T]$, $n,\ m \geq 1$, we establish that
\begin{eqnarray*}
{\E}\left| Y_{t}^{n+m}-Y_{t}^{n}\right| ^{2} \leq
e^{\frac{CT}{1-\alpha}}\left(\frac{1-\alpha}{C}+1\right)\int_{t}^{T_1
}\rho(s,{\E}\left| Y_{s}^{n+m-1}-Y_{s}^{n-1}\right| ^{2})ds,
\end{eqnarray*}
and
\begin{eqnarray*}\label{} {\E}\left|Y_{t}^{n}\right| ^{2}
&\leq&\mu_t^2+M\int_{t}^{T_1}
\rho(s,{\E}\left|Y_{s}^{n-1}\right|^2) ds
\end{eqnarray*}
where
\begin{eqnarray*}\mu_t^2&=&e^{\frac{(2C+\alpha)T}{1-\alpha}}\left({\E}\left|Y_{T_1}\right|
^{2}+{\E}\int_{t}^{T
}\left[\frac{2(1-\alpha)}{2C+\alpha}\left|f(s,0,0)\right|^2\right.\right.\\
&&\ \ \ \ \ \ \ \ \ \ \ \ \ \ \ \left.\left.+
\Big(\frac{1+2C}{1-\alpha}\Big)\left|g(s,0,0)\right|^2\right]ds+
\frac{1}{|\beta|}{\E}\int_{t}^{T}\left|h(s,0)\right|^2dA_s\right),
\end{eqnarray*}
Let
\begin{eqnarray*}M_2=2\mu_0^2&=&2e^{\frac{(2C+\alpha)T}{1-\alpha}}\left({\E}\left|Y_{T_1}\right|
^{2}+{\E}\int_{0}^{T
}\left[\frac{2(1-\alpha)}{2C+\alpha}\left|f(s,0,0)\right|^2
\right.\right.\\&& \ \ \ \ \ \ \ \ \ \ \ \ \ \ \ \ \ \left.\left.+
\Big(\frac{1+2C}{1-\alpha}\Big)\left|g(s,0,0)\right|^2\right]ds
+\frac{1}{|\beta|}{\E}\int_{0}^{T
}\left|h(s,0)\right|^2dA_s\right).
\end{eqnarray*}
We can also find $T_2\in[0,T_1[$ such that
\begin{eqnarray*}
{\E}\left|Y_{t}^{n}\right| ^{2} &\leq&M_2,\ \ n\geq 1, \ t\in
[T_2,T_1].
\end{eqnarray*}
Here $T_2=0$ or $T_2\in]0,T_1[$ such that $\displaystyle\int_{T_2}^{T_1 }
\rho(s,M_2) ds= \frac{\mu_0^2}{M}.$  As above, we prove the
existence of the solution to BDSDE \eqref{2e} on $[T_2,T_1].$ If
$T_2=0$, the proof of the existence is complete. Overwise, we
repeat the above procedures. Thus, we obtain a sequence $\{T_p,\
\mu_t^p,\ M_p,\ \ p\geq 1\}$ defined by
\begin{eqnarray*}
&&0\leq T_p< T_{p-1}<...<T_1<T_0=T,\\
&&\mu_t^p=e^{\frac{(2C+\alpha)T}{1-\alpha}}\left[{\E}\left|Y_{T_{p-1}}\right|
^{2}+{\E}\int_{t}^{T
}\left(\frac{2(1-\alpha)}{2C+\alpha}\left|f(s,0,0)\right|^2 \right.\right.\\
&&\left.\left.\ \ \ \ \ \ \ \ \ \ \ \ \ \ \  \ \ \ \ \ +
\left(\frac{1+2C}{1-\alpha}\right)\left|g(s,0,0)\right|^2\right)ds+
\frac{1}{|\beta|}{\E}\int_{t}^{T
}\left|h(s,0)\right|^2dA_s\right],\\
&&M_p=2\mu_0^p=2e^{\frac{(2C+\alpha)T}{1-\alpha}}\left[{\E}\left|Y_{T_{p-1}}\right|
^{2} +{\E}\int_{0}^{T
}\left(\frac{2(1-\alpha)}{2C+\alpha}\left|f(s,0,0)\right|^2\right.\right. \\
&&\left.\left.\ \ \ \ \ \ \ \ \ \ \ \ \ \ \  \ \ \ \ \  +
\left(\frac{1+2C}{1-\alpha}\right)\left|g(s,0,0)\right|^2\right)ds+
\frac{1}{|\beta|}{\E}\int_{0}^{T
}\left|h(s,0)\right|^2dA_s\right],\\
&&\text{\ \ and \ \ } \int_{T_{p}}^{T_{p-1} } \rho(s,M_p) ds=
\frac{\mu_0^p}{M}.
\end{eqnarray*}Therefore, by iteration, we deduce the existence of
a solution to BDSDE \eqref{eq0} on $[T_p,T].$\\
Finally, setting
$$A=2e^{\frac{(2C+\alpha)T}{1-\alpha}}\left[{\E}\int_{0}^{T
}\left(\frac{2(1-\alpha)}{2C+\alpha}\left|f(s,0,0)\right|^2+
\left(\frac{1+2C}{1-\alpha}\right)\left|g(s,0,0)\right|^2\right)ds+
\frac{1}{|\beta|}{\E}\int_{0}^{T
}\left|h(s,0)\right|^2dA_s\right]$$
and using the same argument as in \cite{MNJMO2}, we prove the existence of
a finite $p\geq 1$ such that $T_p=0$. Thus, we obtain
the existence of the solution on $[0,T].$
\\

\noindent {\it Uniqueness.} \ Let $\left(Y,Z\right) \ ,
\left(Y',Z'\right)\in \mathcal{S}^2([0,T];
{\R}^k)\times\mathcal{M}^2(0,T; {\R}^{k\times d})$ be two
solutions of  BDSDE \eqref{eq0}.\\ Let $\theta>0$. By virtue of
It\^o's formula, we have
\begin{eqnarray*}\label{}
&&{\E}|Y_t-Y'_t|^2e^{\theta
t}+\theta{\E}\int_{t}^{T}|Y_s-Y'_s|^2e^{\theta
s}ds+{\E}\int_{t}^{T}\|Z_s-Z'_s\|^2e^{\theta
s}ds\\&&=2{\E}\int_{t}^{T}\left\langle Y_s-Y'_s,
f(s,Y_s,Z_s)-f(s,Y'_s,Z'_s)\right\rangle e^{\theta s}ds+
2{\E}\int_{t}^{T}\left\langle Y_s-Y'_s,
h(s,Y_s)-h(s,Y'_s)\right\rangle e^{\theta s}dA_s\\&&+
{\E}\int_{t}^{T}|g(s,Y_s,Z_s)-g(s,Y'_s,Z'_s)|^2e^{\theta s}ds.
\end{eqnarray*}
By virtue of the assumption {\bf (H3)}, {\bf (H4)} and Young's inequality
$2ab\leq \frac{1}{\theta}a^2+\theta b^2 $, we derive
\begin{eqnarray*}\label{}
&&{\E}|Y_t-Y'_t|^2e^{\theta
t}+(1-\alpha-\frac{1}{\theta} C){\E}\int_{t}^{T}\|Z_s-Z'_s\|^2e^{\theta
s}ds+2|\beta|{\E}\int_{t}^{T}|Y_s-Y'_s|^2e^{\theta
s}dA_s\\&&\leq
\left(\frac{1}{\theta}+1\right){\E}\int_{t}^{T}\rho(s,|Y_s-Y'_s|^2)e^{\theta
s}ds.
\end{eqnarray*}
Choosing $\theta>\frac{C}{1-\alpha}$ and noting that $1\leq e^{\theta t}\leq e^{\theta
T}$, \ $\forall\ t\in[0,T]$, we get
\begin{eqnarray}\label{ST}
&&{\E}|Y_t-Y'_t|^2+(1-\alpha-\frac{1}{\theta} C){\E}\int_{t}^{T}\|Z_s-Z'_s\|^2ds+
2|\beta|{\E}\int_{t}^{T}|Y_s-Y'_s|^2dA_s\\&&\leq
\left(\frac{1}{\theta}+1\right)e^{\theta
T}{\E}\int_{t}^{T}\rho(s,|Y_s-Y'_s|^2)ds \notag.
\end{eqnarray} Therefore
\begin{eqnarray*}\label{}
{\E}|Y_t-Y'_t|^2\leq \left(\frac{1}{\theta}+1\right)e^{\theta
T}\int_{t}^{T}\rho(s,{\E}|Y_s-Y'_s|^2)ds.
\end{eqnarray*}
In view of the comparison Theorem for ODE, we have
\begin{eqnarray*}\label{}
{\E}|Y_t-Y'_t|^2\leq r(t), \ \ \forall\ t\in[0,T],
\end{eqnarray*}
where $r(t)$ is the maximum left shift solution of the following
equation:
$$\left\{
\begin{array}{ccc}
  u' & = & -(\frac{1}{\theta}+1)e^{\theta
T}\rho(t,u); \\
  u(T) & = & 0. \text{ \ \ \ \ \ \ \ \ \ \ \ \ \ \ \ \ \ \ \ \ \ \ \ \ \ \  } \\
\end{array}
\right.$$ By virtue of the assumption {\bf(H3)}, $r(t)=0$,
$t\in[0,T]$. Thus \ ${\E}|Y_t-Y'_t|^2=0$, \ $t\in[0,T]$, this
means $Y_t=Y'_t$, a.s.. It then follows from \eqref{ST} that
$Z_t=Z'_t$, a.s., for any $t\in[0,T]$.
\end{proof}

\end{document}